\documentclass[11pt]{amsart}
\usepackage{amsfonts,amssymb,amsmath,euscript}
\usepackage{hyperref}

\theoremstyle{plain}
 \newtheorem{thm}{Theorem}[section]

\numberwithin{equation}{section}
\theoremstyle{definition}

\newcommand{\mbR}{\mathbb{R}}
\newcommand{\mbC}{\mathbb{C}}

\newcommand{\clH}{\mathcal{H}}
\newcommand{\clK}{\mathcal{K}}

\newcommand{\euE}{{\EuScript E}}
\newcommand{\euH}{{\EuScript H}}
\newcommand{\Tr}{\operatorname{Tr}}

\newcommand{\sgn}{\operatorname{sgn}}
\renewcommand{\Im}{\operatorname{Im}}
\renewcommand{\Re}{\operatorname{Re}}
\newcommand{\hlambda}{{\mathfrak{h}_\lambda}}

\newcommand{\dom}{\operatorname{dom}} 
\newcommand{\ran}{\operatorname{ran}}
\newcommand{\rank}{\operatorname{rank}}
\newcommand{\scal}[1]{\left\langle #1 \right\rangle}
\newcommand{\xia}{\xi^{(a)}}
\newcommand{\xis}{\xi^{(s)}}

\newcommand{\ind}{\operatorname{ind}}

\begin{document}
\title{A remark on the imaginary part of resonance points}
\author{Nurulla Azamov and Tom Daniels}
\address{College of Science and Engineering
   \\ Flinders University 
   \\ South Rd, Tonsley, SA 5042 Australia}
\email{nurulla.azamov@flinders.edu.au}
\email{tom.daniels@flinders.edu.au}


%\begin{center} \tiny \sf\today \ v1.3 \end{center}

\begin{abstract}
In this paper we prove for rank one perturbations that 
%negative two times reciprocal of the imaginary part of resonance point
%is equal to the rate of change of the scattering phase as a function of the coupling constant,
%where the coupling constant is equal to the real part of the resonance point.
the imaginary part of a resonance point is inversely proportional by a factor of $-2$ to the rate of change of the scattering phase, as a function of the coupling parameter, evaluated at the real part of the resonance point.
This equality is in agreement with the Breit-Wigner formula from quantum scattering theory. 
For more general relatively trace class perturbations, 
we also give a formula for the spectral shift function in terms of resonance points, non-real and real.
\end{abstract}
% \begin{center} 
%    \tiny \sf\today \quad \bigskip v\,1.3
% \end{center}
\maketitle

\section{Introduction}
Given a self-adjoint operator~$H_0$ and a relatively compact self-adjoint operator~$V,$
a resonance point $r_z$ of the triple $(z; H_0,V)$ can be defined as a pole of the meromorphic operator valued function %$s \mapsto VR_z(H_s)$
\[
  s \mapsto R_z(H_s) = R_z(H_0)(1 + sVR_z(H_0))^{-1},
\]
where $H_s = H_0 + sV$ and $R_z(H) = (H - z)^{-1}.$ 
%This function is meromorphic by the analytic Fredholm alternative (see e.g.~\cite[Theorem~?]{Yaf}).
We stress that this is in contrast to the interpretation of resonances as poles of the resolvent as a function of energy~$z.$
Resonance points, considered as functions of~$z,$ are branches of multivalued analytic functions (of Herglotz type).
Under certain conditions on the pair $(H_0,V),$ which ensure the existence of scattering theory (namely the limiting absorption principle, see e.g.~\cite[Chapter 6]{Yaf92}), 
resonance points $r_z$ have limit values $r_{\lambda+i0}$ for a.e.~$\lambda \in \mbR.$ 
In case the resonance point $r_{\lambda+i0}$ is real, it has several interpretations, as discussed in detail in the introduction of~\cite{Aza16}. 
One such interpretation is that one of the scattering phases $\theta_j(\lambda;r)$ of the scattering matrix $S(\lambda; H_r, H_0),$
considered as an analytic function of the coupling parameter~$r,$
suffers a sudden jump by an integer multiple of~$2\pi$ when~$r$ crosses the real resonance point~$r_\lambda$ 
(in fact such a jump is only revealed when~$\lambda$ is perturbed slightly to $\lambda+i\epsilon$).
By a scattering phase $\theta_j(\lambda;r)$ we mean that $e^{i\theta_j(\lambda;r)}$ is an eigenvalue 
of the scattering matrix $S(\lambda; H_r, H_0).$

The limit values of resonance points $r_{\lambda+i0}$ with non-zero imaginary parts have not been investigated elsewhere. 
In this paper for rank-one perturbations, in which case there is only 
one non-zero scattering phase $\theta_1(\lambda;r),$ we prove the formula:
\begin{equation} \label{F: the formula}
  \frac{\partial\theta_1(\lambda;r)}{\partial r}\Big|_{r=\Re r_{\lambda+i0}} 
       = - \frac {2}{\Im r_{\lambda+i0}}, \quad \text{a.e.} \ \lambda\in\mbR. 
\end{equation}

Since $r_{\lambda+i0}$ is a pole of the scattering matrix, this formula is in agreement with the Breit-Wigner formula from quantum scattering theory, see e.g.~\cite[Chapter XVIII]{Boh},~\cite[Chapter 13]{Tay},
with the difference that the phase is considered as a function of the coupling parameter instead of energy. 
%Indeed, considering the scattering matrix as a function of energy, the imaginary part of a pole and the rate of change of the scattering phase at that resonant energy are known by physicists to be inversely proportional.

Also established here is a more general formula, 
which applies in the case that~$V$ is a certain kind of relatively trace class perturbation. 
For comparison, some analogous considerations in terms of the energy parameter appear in~\cite[\S9.3]{BirYaf} along with further references.

\section{Proof}

\begin{thm}\label{T: rank one case} 
If~$H_0$ is a self-adjoint operator and~$V$ is a rank one self-adjoint operator, then 
for a.e.~$\lambda$ the formula~\eqref{F: the formula} holds.
In this formula, for $z$ outside the essential spectrum of~$H_0,$ 
the number $r_z$ is the unique pole of the meromorphic function 
$
  \mbC \ni s \mapsto VR_z(H_s) %V(H_0+sV-z)^{-1},
$ 
with $r_{\lambda+i0} =\lim_{y \to 0^+} r_{\lambda+iy},$ and $e^{i\theta_1(\lambda;r)}$ is the 
unique non-trivial eigenvalue of the scattering matrix~$S(\lambda; H_r,H_0).$ 
\end{thm}

Before proceeding to the proof we discuss its context and sketch an important lemma whose details lie outside the scope of this paper. 
Since~$V$ has rank one, the default premise of this theorem -- the limiting absorption principle -- holds; 
this is necessary for the existence of the scattering matrix and therefore also the scattering phase.
Because the proof involves considering these objects as functions of the coupling parameter, it requires the constructive approach to stationary scattering theory given in~\cite{Aza11,AzaDan18} and outlined in the introduction to~\cite{Aza16}. 
In this approach, objects such as the wave matrices $w_\pm(\lambda; H_r,H_0)$ and scattering matrix $S(\lambda; H_r,H_0)$ are defined by explicit formulas for all $r$ except a discrete set, as long as $\lambda$ belongs to a pre-defined set of full Lebesgue measure in~$\mbR,$ which comes from the limiting absorption principle. 
Moreover, the scattering matrix can be differentiated with respect to~$r,$ as discussed further below.
%Since $V$ has rank one, it can be easily shown that the resonance point $r_z$ is a well-defined single-valued Herglotz function in the upper complex half-plane~$\mbC_+$
%(to be precise, $r_z$ is a Herglotz function if $V$ is non-negative, otherwise $-r_z$ is a Herglotz function). 

The proper setting for the proof of Theorem~\ref{T: rank one case} and its generalisation Theorem~\ref{T: sum theta=int beta}, is outlined by the following assumptions which are common within scattering theory.
\begin{enumerate}
  \item[$(1)$] $H_0$ is a self-adjoint operator on a (separable complex) Hilbert space~$\clH.$ 
  \item[$(2)$] $V$ is a symmetric form on~$\clH,$ which admits the decomposition $V = F^*JF$ on the form domain of $H_0,$ i.e.
  \[
   V\colon(f,g)\mapsto \scal{Ff,JFg}, \quad f,g\in\dom |H_0|^{1/2},
  \] 
  where $F \colon \clH \to \clK$ is a closed operator 
%  with zero kernel and co-kernel,
  and $J$ is a self-adjoint bounded operator
  on~$\clK.$
  \item[$(3)$] The sandwiched resolvent $T_z(H_0) = FR_z(H_0)F^*,$ 
  where $R_z(H) = (H-z)^{-1}$ is the resolvent of~$H,$ is (or more precisely extends to) a compact operator
  for some (and thus for any) $z \notin \sigma(H_0),$ the spectrum of~$H_0.$
  \item[$(4)$] $F$ is bounded or $H_0$ is semi-bounded.
%  \item[$(5)$] The uniform limit $T_{\lambda+i0}(H_0) := \lim_{y \to 0^+} T_{\lambda+iy}(H_0)$ exists.
\end{enumerate}
%Without loss of generality we assume that~$F$ has trivial kernel (by extending to the kernel as any compact operator) and that $F\dom |H_0|^{1/2}$ is dense in~$\clK.$ 
These conditions imply that the perturbed operator $H_r := H_0 + rV,$ $r\in\mbR,$ is well-defined, as an operator-sum if $F$ is bounded or a form-sum if $H_0$ is semi-bounded. 
In fact we will need the following strengthened version of condition~$(3).$
\begin{enumerate}
  \item[$(3')$] The sandwiched resolvent $T_z(H_0) = FR_z(H_0)F^*$ is (extends to) a trace class operator for some (and thus any) $z \notin \sigma(H_0).$
\end{enumerate}
Given $(3'),$ it follows from the second resolvent identity that~$T_z(H_r)$ also belongs to the trace class for any $z\in\mbC\setminus\mbR.$ 
This condition implies the limiting absorption principle in the following form: 
the set of points~$\lambda\in\mbR$ for which the uniform limit $T_{\lambda+i0}(H_r) := \lim_{y \to 0^+} T_{\lambda+iy}(H_r)$ exists, has full Lebesgue measure in~$\mbR.$ 
In addition, the limit of the imaginary part $\Im T_{\lambda+i0}(H_r) = \lim_{y\to 0^+}\Im T_z(H_r)$ exists in the trace class norm for a.e.~$\lambda\in\mbR.$
The full set of points~$\lambda$ for which both of these limits exist will be denoted by~$\Lambda(H_r,F).$

If $V$ is relatively compact with respect to~$H_0$ then $F = \sqrt{|V|}$ and $J = \sgn V$ satisfy the conditions $(1)$--$(3).$ 
In this case for~$z\in\mbC\setminus\mbR,$ the compact operators $VR_z(H_0)$ and $JT_z(H_0)$ share the same non-zero eigenvalues and it follows that resonance points~$r_z$ corresponding to~$z$ can be defined as poles of the meromorphic function
\[
 s\mapsto T_z(H_s) = T_z(H_0)(1 + sJT_z(H_0))^{-1},
\]
extending the definition given above.
This definition also makes sense for~$z=\lambda+i0$ provided~$\lambda$ belongs to~$\Lambda(H_0,F).$
What's more, for~$\lambda\in\Lambda(H_0,F),$ it happens that $\lambda\in\Lambda(H_r,F)$ if and only if~$r\in\mbR$ is non-resonant at~$\lambda.$

Note that if $V$ is finite-rank, then so can be the choice of $F,$ hence in this case conditions~$(3')$ and~$(4)$ are also satisfied.

\smallskip
We now briefly review how the scattering matrix can be realised as a function of the coupling parameter by taking a constructive approach to stationary scattering (for more information see~\cite{Aza11,AzaDan18} and the introduction to~\cite{Aza16}). 
A~fibre Hilbert space~$\hlambda(H_0)$ is defined for any $\lambda\in\Lambda(H_0,F)$ as the closed range
\[
 \hlambda(H_0) = \mathrm{cl}\left(\ran\sqrt{\Im T_{\lambda+i0}(H_0)}\right) \subset \clK
\]
and these give rise to the direct integral 
\begin{align*}
 \euH(H_0) &:= \int_{\Lambda(H_0,F)}^\oplus\hlambda(H_0)\,d\lambda 
 \\&= \big\{f\in L_2(\Lambda(H_0,F),\clK) : f(\lambda)\in\hlambda(H_0)\, \text{ for a.e. } \lambda\in\Lambda(H_0,F) \big\}.
\end{align*}
The {\em evaluation operator} $\euE_\lambda(H_0)$ is defined on the range of~$F^*$ by
\[
 \euE_\lambda(H_0) = \sqrt{\pi^{-1}\Im T_{\lambda+i0}(H_0)}(F^*)^{-1}
\] 
and the collection $\euE(H_0) = \{\euE_\lambda(H_0) : \lambda\in\Lambda(H_0,F)\},$ considered as an operator from~$\clH$ to~$\euH(H_0)$ defined on~$\ran F^*,$ extends to a partial isometry which diagonalises the absolutely continuous part of~$H_0$ (\cite[Theorem~5.1]{AzaDan18}, \cite[Theorem~3.4.2]{Aza11}):
\[
 \euE(H_0)\colon \clH \to \euH(H_0), \qquad H_0^{(a)} = \euE^*(H_0)M_{\lambda}\euE(H_0),
\]
where $M_\lambda$ denotes the operator of multiplication by~$\lambda.$
For any non-resonant~$r,$ the wave matrices $w_{\pm}(\lambda;H_r,H_0)$ may then be defined as unitary transforms from $\hlambda(H_0)$ to $\hlambda(H_r),$ which are uniquely determined for $f,g\in\ran F^*$ by 
\[
 \scal{\euE_\lambda(H_r)f, w_\pm(\lambda;H_r,H_0)\euE_\lambda(H_0)g} = \lim_{y\to 0^+}\frac y\pi \scal{R_{\lambda+iy}(H_r)f,R_{\lambda+iy}(H_0)g}.
\]
The scattering matrix $S(\lambda;H_r,H_0) = w^*_+(\lambda;H_r,H_0)w_-(\lambda;H_r,H_0)$ can be shown to satisfy the stationary formula 
\begin{equation}\label{F: stat formula}
 S(\lambda;H_r,H_0) = 1 - 2ir\sqrt{\Im T_{\lambda+i0}(H_0)}J(1 + rT_{\lambda+i0}(H_0)J)^{-1}\sqrt{\Im T_{\lambda+i0}(H_0)},
\end{equation}
for any $\lambda\in\Lambda(H_0,F)$ and non-resonant~$r.$
This allows the scattering matrix, for fixed such~$\lambda,$ to be considered as a function of~$r.$ 
Although its factor $(1 + rT_{\lambda+i0}(H_0)J)^{-1}$ is meromorphic with poles at resonance points, since the scattering matrix is unitary and hence bounded for non-resonant $r\in\mbR,$ it admits analytic continuation to a neighbourhood of the real axis.

A significant part of the proof of Theorem~\ref{T: rank one case} (see~\cite[Theorem~5.7]{AzaDan18}, \cite[Theorem~7.3.3]{Aza11}) is the following fact, obtained from~\eqref{F: stat formula}.
The derivative of the scattering matrix at any non-resonant~$r$ is given by
\begin{multline}\label{F: S'(r)}
  \frac {d S(\lambda; H_r, H_0)}{dr} \\ = - 2 i\, w_+(\lambda; H_0,H_r)
  \sqrt{\Im T_{\lambda+i0}(H_r)}J\sqrt{\Im T_{\lambda+i0}(H_r)} \,
  w_+(\lambda; H_r, H_0)S(\lambda; H_r, H_0).
\end{multline}
Assuming condition~($3'$), this derivative can be taken in the trace class norm. 

\begin{proof}[Proof of Theorem~\ref{T: rank one case}]
Let $\lambda$ be a real number from the full set $\Lambda(H_0,F).$
Since~$\lambda$ is fixed, we write~$S(r)$ for $S(\lambda;H_r,H_0)$ and $\theta_j(r)$ for~$\theta_j(\lambda;r).$
Since the scattering matrix is unitary, it follows from~\eqref{F: S'(r)} that 
\begin{equation}\label{F: infinitesimal SM}
  S'(r)S^{-1}(r) = - 2 i\, w_+(\lambda; H_0,H_r) 
\sqrt{\Im T_{\lambda+i0}(H_r)}J\sqrt{\Im T_{\lambda+i0}(H_r)}\, 
w_+(\lambda; H_r,H_0).
\end{equation}
Further, since 
$
 w_+(\lambda; H_r,H_0)w_+(\lambda; H_0,H_r) = 1_{\hlambda(H_r)}
$
(see~\cite[Theorem~5.3]{AzaDan18}, \cite[Corollary~5.3.8]{Aza11}), 
taking traces of both sides of the equality~\eqref{F: infinitesimal SM} gives 
\begin{equation}\label{F: Tr(EVE)}
  \Tr\left(S'(r)S^{-1}(r)\right) = - 2 i \Tr\left(\sqrt{\Im T_{\lambda+i0}(H_r)}J\sqrt{\Im T_{\lambda+i0}(H_r)}\right) 
\end{equation}
The trace on the right can be interpreted to be associated to the trace class operators on the whole space~$\clK,$ rather than the fibre Hilbert space~$\hlambda(H_r),$ since $\ker\sqrt{\Im T_{\lambda+i0}(H_r)} = \hlambda(H_r)^\perp.$
It follows using condition $(3')$ that the equality~\eqref{F: Tr(EVE)} can rewritten as
\begin{align}
  \Tr\left(S'(r)S^{-1}(r)\right) & = - 2 i \Tr\left(J\Im T_{\lambda+i0}(H_r)\right) \nonumber
  \\ &= - \lim_{y\to 0^+} 2 i\Tr\left(J\Im T_{\lambda+iy}(H_r)\right) \nonumber
  \\ &= - \lim_{y\to 0^+} \Tr\left(JT_{\lambda+iy}(H_r)\right) - \Tr\left(JT_{\lambda-iy}(H_r)\right). \label{F: Tr(Im A)}
\end{align}
Since by the premise $V$ has rank 1, in this case $\Tr(JT_z(H_0)) = \Tr(VR_z(H_0)).$
We recall that a resonance point $r_z$ corresponding to the triple $(z; H_0,V)$ is a complex number such that $(r-r_z)^{-1}$ is an eigenvalue of the compact operator~$VR_z(H_r).$ 
%(see~\cite[\S 3.1]{Aza16}, more specifically~\cite[(3.1.2)]{Aza16}).
In this case the operator~$VR_z(H_r)$ has only one eigenvalue. 
Therefore, with $z=\lambda+iy$ and using the fact that $\bar r_{z} = r_{\bar z},$ we find that 
\begin{equation*}
  \begin{split}
     \Tr\left(S'(r)S^{-1}(r)\right) & = - \lim_{y\to 0^+} \left((r-r_{z})^{-1} - (r-\bar r_{z})^{-1}\right)
           \\ & = - \lim_{y\to 0^+} \frac {-\bar r_{z} + r_{z}}{(r-r_{z})(r-\bar r_{z})}
           \\ & = - \lim_{y\to 0^+} \frac {2i \Im r_{z}}{(r-r_{z})(r-\bar r_{z})}.
  \end{split}
\end{equation*}
Taking the limit and replacing $r$ by $\Re r_{\lambda+i0}$ (assuming it is not resonant, i.e.~$\Im r_{\lambda+i0}\neq 0$), 
\begin{equation*}
%  \begin{split}
     \Tr\left(S'(r)S^{-1}(r)\right)\big|_{r = \Re r_{\lambda+i0}}
%         & = - \frac {2i \Im r_{\lambda+i0}}{(\Re r_{\lambda+i0}-r_{\lambda+i0})(\Re r_{\lambda+i0}-\bar r_{\lambda+i0})}
       = - \frac {2i}{\Im r_{\lambda+i0}}.
%  \end{split}
\end{equation*}

On the other hand, since $\rank(V) = 1,$ the scattering matrix~$S(r)$ is one-dimensional. 
Hence, it is the operator of multiplication by its eigenvalue:
$
  S(r) = e^{i\theta_1(r)}\cdot 1_{\hlambda(H_0)}.
$
It follows that 
\begin{equation*}
  \begin{split}
     \Tr\left(S'(r)S^{-1}(r)\right)\big|_{r = \Re r_{\lambda+i0}} 
           & = \frac{d e^{i\theta_1(r)}}{dr} e^{-i\theta_1(r)}\big|_{r = \Re r_{\lambda+i0}}
        \\ & = i\theta'_1(\Re r_{\lambda+i0}).
  \end{split}
\end{equation*}
Comparing the last two formulas completes the proof. 
\end{proof}

In the proof of Theorem~\ref{T: rank one case} the following equality is derived:
\[
    \theta'_1(\lambda;r)
%      = - \frac {2 \Im r_{\lambda+i0}}{(r-r_{\lambda+i0})(r-\bar r_{\lambda+i0})}
      = - \frac {2\beta }{|r-\alpha|^2 + |\beta|^2}, 
\]
where $\alpha := \Re r_{\lambda+i0},$ and  $\beta := \Im r_{\lambda+i0}.$
If the phase $\theta_1(\lambda;r)$ is chosen so that $\theta_1(\lambda;0)=0,$ then integrating gives the formula
\[
 \theta_1(\lambda;1) = - \int_0^1 \frac {2\beta }{|r-\alpha|^2 + |\beta|^2}\,dr.
\]

Scattering phases are well-known to be closely related to the spectral shift function (SSF) $\xi(\lambda;H_1,H_0)$ (see e.g.~\cite{BirYaf,Pus01}).
In this regard is the following formula for the absolutely continuous~SSF $\xia(\lambda;H_1,H_0)$ (see~\cite[\S5.3]{AzaDan18}, \cite[Theorem~9.2.2]{Aza11}; cf.~\cite{Pus01})
\begin{equation}\label{F: xia = sum theta}
  \xia(\lambda; H_1,H_0) = -\frac 1{2\pi} \sum_{j=1}^\infty \theta_j(\lambda; 1),
\end{equation}
where $e^{i\theta_j(\lambda;r)},$ $r\in[0,1],$ are the eigenvalues of the scattering matrix~$S(\lambda;H_r,H_0),$ which are continuously enumerated with phases chosen so that $\theta_j(\lambda;0)=0.$  
On the other hand, the singular~SSF $\xis(\lambda;H_1,H_0)$ is known to be given in terms of resonance points by the {\em total resonance index} (see \cite[Theorem~1.3]{AzaDan18}, \cite[Theorem~6.3.2]{Aza16})
\begin{equation}\label{F: xis = res ind}
 \xis(\lambda;H_1,H_0) = \sum_{r_\lambda \in [0,1]} \ind_{res}(\lambda; H_{r_\lambda},V),
\end{equation}
which is the difference $N_+ - N_-$ in the numbers $N_\pm$ of resonance points which converge to the interval~$[0,1]$ from the half-plane~$\mbC_\pm.$

The following generalisation of Theorem~\ref{T: rank one case} shows that the absolutely continuous SSF can also be expressed in terms of resonance points.

\begin{thm}\label{T: sum theta=int beta}
Let $H_0$ and $V = F^*JF$ satisfy conditions~$(1),$~$(2),$~$(3'),$ and~$(4),$ and let $\theta_j(\lambda;1)$ be as in~\eqref{F: xia = sum theta}.
Then for a.e.~$\lambda\in\mbR,$
\begin{equation} \label{F: sum theta=int beta}
  \sum_{j=1}^\infty \theta_j(\lambda; 1) 
       = - \int_0^1 \sum_{j=1}^\infty \frac {2\beta_j}{|r-\alpha_j|^2 + |\beta_j|^2}\,dr,
\end{equation}
where $r^j_{\lambda+i0} = \alpha_j + i \beta_j$ is the $j$th resonance point with non-zero $\beta_j,$ 
that is, $(r-r^j_{\lambda+i0})^{-1}$ is the $j$th eigenvalue of $JT_{\lambda+i0}(H_r).$ 
\end{thm}

Combining~\eqref{F: xia = sum theta} and \eqref{F: sum theta=int beta} gives
\[
  \xia(\lambda; H_b,H_a) 
      = \frac 1{2\pi} \int_a^b \sum_{j=1}^\infty \frac {2\beta_j}{|r-\alpha_j|^2 + |\beta_j|^2}\,dr.
\]
To obtain a formula for the SSF, we can add~\eqref{F: xis = res ind}:
\begin{equation*}\label{F: SSF in terms of resonance data}
  \xi(\lambda; H_1,H_0) 
      = \frac 1{2\pi} \int_0^1 \sum_{j=1}^\infty \frac {2\beta_j}{|r-\alpha_j|^2 + |\beta_j|^2}\,dr
    + \sum_{r_\lambda \in [0,1]} \ind_{res}(\lambda; H_{r_\lambda},V).
\end{equation*}
This shows that resonance points $r_{\lambda+i0}^j$ with zero imaginary part $\beta_j$ still contribute to the SSF in the form of resonance indices. 

\begin{proof}[Proof of Theorem~\ref{T: sum theta=int beta}]
The equality~\eqref{F: Tr(Im A)} holds by the same argument as before.
Since $JT_z(H_r)$ is trace class for $z=\lambda+iy,$ $y>0,$ its eigenvalues $(r-r^j_z)^{-1}$ are summable and we have
\begin{align*}
 \Tr\left(S'(r)S^{-1}(r)\right) &= - \lim_{y\to 0^+}\Bigg( \sum_{j=1}^\infty(r-r^j_z)^{-1} - \sum_{j=1}^\infty(r-\bar r^j_z)^{-1}\Bigg)
 \\ &= - \lim_{y\to 0^+}\Bigg( \sum_{j=1}^\infty \frac {2i \Im r^j_{z}}{(r-r^j_{z})(r-\bar r^j_{z})} \Bigg)
 \\ &= - \sum_{j=1}^\infty \frac {2i\beta_j}{|r-\alpha_j|^2 + |\beta_j|^2},
\end{align*}
where the interchange of limits in the last equality can be justified using the fact that the sums involved are uniformly bounded. 
Indeed, each sum is equal to the trace of $2iJ\Im T_z(H_r),$ which converges to $2i J\Im T_{\lambda+i0}(H_r)$ in the trace class. 
By integrating, it remains to show that
\begin{equation*}
 \int_0^1 \Tr\left(S'(r)S^{-1}(r)\right) dr =  i \sum_{j=1}^\infty \theta_j(\lambda;1).
\end{equation*}
This can be seen as a consequence of the fact that, when divided by $-2\pi i,$ both sides are known representations of the absolutely continuous SSF (see~\cite[Theorem~3.2 and \S5.3]{AzaDan18}, \cite[\S\S8-9]{Aza11}):
\begin{align*}
\xia(\lambda;H_b,H_a) &= \frac{1}{\pi} \int_0^1 \Tr\left(J\Im T_{\lambda+i0}(H_r)\right) dr
 \\&= -\frac{1}{2\pi} \sum_{j=1}^\infty \theta_j(\lambda;1). \qedhere
\end{align*}
\end{proof}

Following from condition~$(3')$ and the proof of Theorem~\ref{T: sum theta=int beta} is the equality
\begin{equation}\label{F: TrImA=ImTrA}
 \Tr (\Im A_{\lambda+i0}(H_r)) = \sum_{j=1}^\infty \Im\sigma_j(r),
\end{equation}
where $\sigma_j(r)$ are the eigenvalues of the operator $A_{\lambda+i0}(H_r) := JT_{\lambda+i0}(H_r).$ 
The root vectors of~$A_{\lambda+i0}(H_r)$ do not depend on the choice of non-resonant~$r$ (see \cite[Proposition~3.1.2]{Aza16}) and we note that if $J=1,$ which may be assumed in the case that $V\geq 0,$ the equality~\eqref{F: TrImA=ImTrA} implies that the system of root vectors is complete; in fact these are equivalent conditions by~\cite[Theorem~V-2.1]{GohKre}.

\end{document}